\documentclass[11pt,english]{article}
\usepackage{lmodern}

\usepackage[T1]{fontenc}
\usepackage[latin9]{inputenc}
\usepackage[letterpaper]{geometry}
\usepackage{babel}
\usepackage{amsthm}
\usepackage{amsmath}
\usepackage{amssymb}
\usepackage{graphicx}
\usepackage[unicode=true,pdfusetitle,
 bookmarks=true,bookmarksnumbered=false,bookmarksopen=false,
 breaklinks=false,pdfborder={0 0 1},backref=false,colorlinks=false]
 {hyperref}

\makeatletter
\numberwithin{equation}{section}
\numberwithin{figure}{section}
\numberwithin{table}{section}
\newcommand{\lyxaddress}[1]{
\par {\raggedright #1
\vspace{1.4em}
\noindent\par}
}
\theoremstyle{plain}
\newtheorem{thm}{\protect\theoremname}
  \theoremstyle{plain}
  \newtheorem{lem}[thm]{\protect\lemmaname}

\makeatother

  \providecommand{\lemmaname}{Lemma}
\providecommand{\theoremname}{Theorem}

\begin{document}

\title{Error Analysis of Finite Differences and the Mapping Parameter in
Spectral Differentiation}

\date{$\:$}

\author{Divakar Viswanath }

\maketitle

\lyxaddress{{\small{}Department of Mathematics, University of Michigan (divakar@umich.edu). }}
\begin{abstract}
The Chebyshev points are commonly used for spectral differentiation
in non-periodic domains. The rounding error in the Chebyshev approximation
to the $n$-the derivative increases at a rate greater than $n^{2m}$
for the $m$-th derivative. The mapping technique of Kosloff and Tal-Ezer
(\emph{J. Comp. Physics}, vol. 104 (1993), p. 457-469) ameliorates
this increase in rounding error. We show that the argument used to
justify the choice of the mapping parameter is substantially incomplete.
We analyze rounding error as well as discretization error and give
a more complete argument for the choice of the mapping parameter.
If the discrete cosine transform is used to compute derivatives, we
show that a different choice of the mapping parameter yields greater
accuracy.
\end{abstract}

\section{Introduction}

The Chebyshev points $x_{j}=\cos(j\pi/n)$, $n=0,1,\ldots,n$, are
commonly used to discretize the interval $[-1,1]$ when the boundary
conditions are not periodic. The $m$-th derivative $f^{(m)}(x)$
may be approximated as $\sum_{k=0}^{m}f(x_{k})w_{k,m}$ where $w_{k,m}$
are differentiation weights. The rounding error in the $m$-th derivative
increases faster than $n^{2m}$ (precise asymptotics will be given
in section 2). In contrast, the rounding error error in Fourier spectral
methods increases at the much milder rate of $n^{m}$ \cite{DonSolomonoff1997,KosloffTalEzer1993}
or $n^{m+1}$.

Kosloff and Tal-Ezer \cite{KosloffTalEzer1993} introduced a mapping
technique to control the growth in rounding errors while preserving
spectral accuracy. The central idea is to replace the function $f(x)$
by the function $F(\xi)=f\left(g(\xi)\right)$ where $g:[-1,1]\rightarrow[-1,1]$,
where 
\begin{equation}
g(\xi)=\frac{\arcsin\alpha\xi}{\arcsin\alpha}\label{eq:intro-map-fn-g}
\end{equation}
is a mapping function that depends upon the parameter $\alpha\in[0,1]$.
The grid in $\xi$ is still Chebyshev with $\xi_{j}=\cos(j\pi/n)$,
and is used to define the mapped grid in $x$ as $x_{j}=\xi_{j}$
for $j=0,1,\ldots n$. The derivative is approximated using 
\[
\frac{df}{dx}=\frac{1}{g'(\xi)}\frac{dF}{d\xi}.
\]
The derivative $dF/d\xi$ is obtained using spectral differentiation
at Chebyshev points and then scaled by $1/g'(\xi)$ to obtain $df/dx$.
Higher derivatives are obtained by iteration of this technique.

The points $x_{j}$ converge to Chebyshev and equi-spaced points,
respectively, in the limits $\alpha\rightarrow0$ and $\alpha\rightarrow1$.
For $\alpha$ in-between, and usually quite close to $1$, the grid
is nearly equi-spaced and still retains spectral accuracy. Since the
grid points are not clustered quadratically near the endpoints $\pm1$,
the growth of rounding errors is milder \cite{DonSolomonoff1997,KosloffTalEzer1993}. 

The function $F(\xi)$ will have a singularity in the complex plane,
due to the mapping, even if $f(x)$ is an entire function. Inspection
of (\ref{eq:intro-map-fn-g}) shows that there are singularities at
$\xi=\pm1/\alpha$. If $f(x)$ is an entire function, such as $f(x)=\sin Kx$,
the interpolation error in $F(\xi)$ using Chebyshev points and in
$f(x)$ using the mapped grid are both controlled by the singularity
locations $\pm1/\alpha$. Kosloff and Tal-Ezer \cite{KosloffTalEzer1993}
recommended the choice of $\alpha$ determined by 
\begin{equation}
\left(\frac{1-\sqrt{1-\alpha^{2}}}{\alpha}\right)^{n}=u\label{eq:intro-kt-balance-eqn}
\end{equation}
where $u$ is the desired accuracy. Don and Solomonoff \cite{DonSolomonoff1997}
showed that taking $u$ to be the machine precision leads to accurate
derivatives. We prefer to take $u$ to be the unit roundoff (for double
precision arithmetic, the unit roundoff is $u=2^{-53}$ and the machine
epsilon is $2^{-52}$ \cite{Higham2002}) because $u$ is the quantity
that comes up naturally in rounding error analysis. However, the distinction
between unit roundoff and machine epsilon has no real consequence
in this situation. The solution of (\ref{eq:intro-kt-balance-eqn})
is given by $\alpha=2/(t+1/t)$ with $t=u^{-1/n}$.

A plausible argument for (\ref{eq:intro-kt-balance-eqn}) is that
it balances the discretization error on the left hand side with the
rounding error on the right hand side. Balancing errors is the right
idea, but it begs the question of why the $n^{2m}$ or $n^{2m+1}$
increase in rounding error is not showing up in (\ref{eq:intro-kt-balance-eqn}).
In this article we give a systematic treatment of both rounding and
discretization errors and show that (\ref{eq:intro-kt-balance-eqn})
is still the right equation regardless of the order of the differentiation
$m$. The order of differentiation $m$ introduces prefactors into
both discretization and rounding error, and these cancel off fortuitously
to leave (\ref{eq:intro-kt-balance-eqn}) as the correct equation
for the mapping parameter $\alpha$ regardless of $m$.

Computation of derivatives at Chebyshev points incurs more error when
the discrete cosine transform is used \cite{DonSolomonoff1997}, in
comparison with carefully computed differentiation matrices \cite{BaltenspergerTrummer2003,DonSolomonoff1997}.
However, the discrete cosine transform is much faster. We show that
(\ref{eq:intro-kt-balance-eqn}) can be modified to choose $\alpha$
in a way that yields slightly more accurate derivatives when the discrete
cosine transform is employed.

Sections 2 and 3 present analyses of rounding and discretization errors,
respectively, showing how the pre-factors cancel leading to (\ref{eq:intro-kt-balance-eqn}).
When $n$ is small the total error is dominated by discretization
error and when $n$ is large the total error is dominated by rounding
error. In section 3, we show that the value of $n$ at which the total
error transitions from discretization error to rounding error does
not depend upon $m$, the order of differentiation.

In section 4, we specialize arguments to the mapping (\ref{eq:intro-map-fn-g}).
We consider the slightly more general balancing equation 
\begin{equation}
\left(\frac{1-\sqrt{1-\alpha^{2}}}{\alpha}\right)^{n}=n^{\beta}u\label{eq:intro-balance-general}
\end{equation}
and find that $\beta=0$ is a good choice when accurate differentiation
matrices are used and $\beta=0.5$ is a better choice for the discrete
cosine transform.

\section{Rounding error analysis of finite differencing}

Spectral differentiation at Chebyshev points is a special case of
finite differencing. In this section, we derive rounding error bounds
assuming the method of partial products. The method of partial products
is an efficient way to calculate finite difference weights \cite{SadiqViswanath2014}.
The rounding error bounds here include the errors that arise during
the calculation of finite difference weights. Some quantities that
arise will recur in the analysis of discretization error. Comparison
to rounding error bounds which assume that the finite difference weights
are exact shows that computation of finite difference weights introduces
only a modest amount of error. Finally we give asymptotic estimates
of the error in the limit $n\rightarrow\infty$.

For floating point arithmetic, we mostly follow Higham \cite{Higham2002},
with a few modifications from \cite{NavarreteViswanath2015}. The
axiom of floating point arithmetic is $\text{fl}(x.\text{op}.y)=(x.\text{op}.y)(1+\delta)$
with $|\delta|\leq u$, where $u$ is the unit-roundoff ($2^{-53}$
for double precision arithmetic). To handle the accumulation of rounding
error, we denote $(1+\delta_{1})^{\rho_{1}}(1+\delta_{2})^{\rho_{2}}\ldots(1+\delta_{n})^{\rho_{n}}$,
with each $\rho_{i}$ equal to $+1$, $0$, $-1$ and $|\delta_{i}|\leq u$,
by $1+\theta_{n}$. In our convention, each occurrence of $\theta_{n}$
is local, which means that two occurrences of $\theta_{n}$ , even
in the same equation, are not assumed to be equal. The quantity $\theta_{n}$
stands for any quantity that may be realized as the accumulated relative
error of $n$ or fewer multiplications and divisions. It satisfies
$|\theta_{n}|\leq\gamma_{n}$, where $\gamma_{n}=nu/(1-nu)$, as long
as $nu<1$. Whenever $\gamma_{n}$ occurs, it is implicitly assumed
that $nu<1$.

Computed quantities are hatted. Thus if $s=x_{1}+\cdots+x_{n}$, with
each $x_{i}$ a floating point number, the computed quantity is denoted
$\hat{s}$. If the addition is from left to right, we may write 
\[
\hat{s}=x_{1}(1+\theta_{n-1})+x_{2}(1+\theta_{n-1})+x_{3}(1+\theta_{n-2})+\cdots+x_{n}(1+\theta_{1}).
\]
Conventions stated above allow us to rewrite this as 
\[
\hat{s}=x_{1}(1+\theta_{n-1})+x_{2}(1+\theta_{n-1})+x_{3}(1+\theta_{n-1})+\cdots+x_{n}(1+\theta_{n-1}).
\]
This device will be employed frequently. Notice that it is a mistake
to factor out $(1+\theta_{n-1})$ in the right hand side, because
each $\theta_{n-1}$ is a local variable and two distinct instances
are not necessarily equal. However, we may write $\hat{s}$ as $\sum x_{j}(1+\theta_{n-1})$,
with the assumption that each $\theta_{n-1}$ inside the summation
is different.

\subsection{Bounds for rounding error}

Assume the $n+1$ grid points to be $x_{0},x_{1},\ldots,x_{n}$. The
weight $w_{k,m}$ in the finite difference formula $f^{(m)}(x)=\sum_{k=0}^{n}w_{k,m}f(x_{k})+\text{error}$
is given by 
\begin{equation}
w_{k,m}=\frac{d^{m}\ell_{k}(x)}{dx^{m}}=w_{k}\frac{d^{m}}{dx^{m}}\prod_{j=0,j\neq k}^{n}(x-x_{j}),\label{eq:bnds-wkm-full}
\end{equation}
where $\ell_{k}(x)$ is the Lagrange cardinal function $\prod_{j\neq k}(x-x_{j})/\prod_{j\neq k}(x_{k}-x_{j})$
and $w_{k}$ is the Lagrange weight $1/\prod_{j\neq k}(x_{k}-x_{j})$. 

If we assume $x=0$, by shifting the grid if necessary, then 
\begin{equation}
w_{k,m}=(-1)^{n-m}m!w_{k}S_{n-m}\left(\left\{ x_{0},\ldots,x_{n}\right\} -\left\{ x_{k}\right\} \right),\label{eq:bnds-wkm-atzero}
\end{equation}
where $S_{n-m}$ is the elementary symmetric function of order $n-m$
\cite{SadiqViswanath2014}. The elementary symmetric function $S_{n-m}$
is the sum of $\binom{n}{n-m}$ terms each of which is a product of
a selection of $n-m$ entries out of the $n$ (all grid points excluding
$x_{k}$). $S_{0}$ is defined as $1$. 

In the method of partial products \cite{SadiqViswanath2014}, the
weight $w_{k,m}$ is computed as follows. The polynomials $\prod_{j=0}^{k}(x-x_{j})$
and $\prod_{j=k}^{n}(x-x_{j})$ are denoted by $L_{k}$ and $R_{k}$,
respectively. Define 
\begin{align}
w'_{k,m} & =\text{coeff of \ensuremath{x^{m}}in }L_{k-1}R_{k+1}\nonumber \\
 & =(-1)^{n-m}\sum_{m_{1},m_{2}}S_{k-m_{1}}\left(x_{0},\ldots,x_{k-1}\right)S_{n-k-m_{2}}\left(x_{k+1},\ldots,x_{n}\right),\label{eq:bnds-wkm-convolution}
\end{align}
where the sum is taken over nonzero integers $m_{1},m_{2}$ satisfying
$m_{1}+m_{2}=m$, $k-m_{1}\geq0$, and $n-k-m_{2}\geq0$. The finite
difference weight $w_{k,m}$ is obtained as $m!w_{k}w'_{k,m}$, where
$w_{k}$ is the Lagrange weight at $z_{k}$.

The elementary symmetric functions that appear in (\ref{eq:bnds-wkm-convolution})
are computed by forming the products $L_{k}$ and $R_{k}$, recursively
\cite{SadiqViswanath2014}. In effect the recurrence 
\begin{equation}
S_{N-m}(y_{1},\ldots,y_{N})=\begin{cases}
y_{N}S_{N-1}(y_{1},\ldots,y_{N-1})\quad\text{if \ensuremath{m=0}}\\
S_{N-m}(y_{1},\ldots,y_{N-1})+y_{N}S_{N-m-1}\left(y_{1},\ldots,y_{N-1}\right)\quad\text{if \ensuremath{N>m>0}}\\
1\quad\text{if \ensuremath{m=N}}
\end{cases}\label{eq:bnds-SNm-recurrence}
\end{equation}
is used for the computation of symmetric functions.

To prove an upper bound on the rounding error in computing $\sum_{k=0}^{n}w_{k,m}f(x_{k})$,
we begin with the following lemma.
\begin{lem}
If the recurrence (\ref{eq:bnds-SNm-recurrence}) is used to calculate
$S_{N-m}(y_{1},y_{2},\ldots,y_{N})$, the computed quantity may be
represented as 
\[
\hat{S}_{N-m}=\sum_{i_{1}<\cdots<i_{N-m}}y_{i_{1}}y_{i_{2}}\ldots y_{i_{N-m}}\left(1+\theta_{f(N,m)}\right)
\]
with $f(N,m)=2(N-1)-m$ for $0\leq m\leq N$.\label{lem:bnds-lem1}\end{lem}
\begin{proof}
One may easily verify that $f(1,0)=f(2,0)=f(1,1)=0$ and $f(2,0)=f(2,1)=1$
suffice. If we inductively assume the lemma for $S_{N-m}(y_{1},\ldots,y_{N-1})$
and $S_{N-m-1}(y_{1},\ldots,y_{N-1})$, and apply the floating point
axiom to the recurrence, we get
\[
f(N,m)\leq\max\left(f(N-1,m-1)+1,f(N-1,m)+2\right)
\]
for $N>2$, along with $f(N,N)=0$ and $f(N,0)=1+f(N-1,0)$. It may
be easily verified that $f(N,m)=2(N-1)-m$ satisfies these relations.
\end{proof}
Next we turn to the roundoff analysis of $w'_{k,m}$ computed using
(\ref{eq:bnds-wkm-convolution}). 
\begin{lem}
The computed value of $w'_{k,m}$ may be represented as 
\[
\hat{w}'_{k,m}=(-1)^{n-m}\sum_{i_{1}<\cdots<i_{n-m}}x_{i_{1}}x_{i_{2}}\ldots x_{i_{n-m}}(1+\theta_{2n+1})
\]
where the summation is over $i_{j}\in\left\{ 0,1,\ldots,n\right\} -\left\{ k\right\} $.\label{lem:bnds-lme2}\end{lem}
\begin{proof}
The number of terms in the summation in (\ref{eq:bnds-wkm-convolution})
is at most $m+1$ and each term is formed using a single multiplication.
Therefore we may represent the computed quantity as 
\[
(-1)^{n-m}\sum_{m_{1},m_{2}}\hat{S}_{k-m_{1}}\left(x_{0},\ldots,x_{k-1}\right)\hat{S}_{n-k-m_{2}}\left(x_{k+1},\ldots,x_{n}\right)(1+\theta_{m+1}).
\]
Applying Lemma (\ref{lem:bnds-lem1}) to $\hat{S}_{k-m}$ (with $N=k$
and $m=m_{1}$)and $\hat{S}_{n-k-m_{2}}$(with $N=n-k$ and $m=m_{2}$),
we get a representation of $\hat{w}'_{k,m}$ that completes the proof.
\end{proof}
The following lemma occurs as a part of Higham's rounding error analysis
of the barycentric formula \cite{Higham2004}. 
\begin{lem}
The computed Lagrange weight $\hat{w}_{k}$ is given by 
\[
\hat{w}_{k}=w_{k}(1+\theta_{2n})
\]
where $w_{k}$ is the exact Lagrange weight.\end{lem}
\begin{proof}
The exact Lagrange weight is given by 
\[
w_{k}=\frac{1}{\prod_{j\neq k}(x_{k}-x_{j})}.
\]
The $\theta_{2n}$ in the lemma is a result of $n$ subtractions,
$n-1$ multiplications, and a single division.\end{proof}
\begin{lem}
The computed weight $w_{k,m}$ may be represented as 
\[
\hat{w}_{k,m}=(-1)^{n-m}m!w_{k}\,\sum_{i_{1}<\cdots<i_{n-m}}x_{i_{1}}x_{i_{2}}\ldots x_{i_{n-m}}(1+\theta_{4n+3})
\]
where the summation is over $i_{j}\in\left\{ 0,1,\ldots,n\right\} -\left\{ k\right\} $.\end{lem}
\begin{proof}
The finite-difference weight $w_{k,m}$ is computed as $m!w_{k}w_{k,m}$.
This lemma is proved using the previous two lemmas and incrementing
the subscript of $\theta$ by $2$ to account for multiplication by
$m!$ and $w_{k}$.\end{proof}
\begin{lem}
If the derivative is being approximated at $x=\zeta$, the computed
weight $w_{k,m}$ may be represented as 
\[
\hat{w}_{k,m}=(-1)^{n-m}m!w_{k}\,\sum_{i_{1}<\cdots<i_{n-m}}\left(x_{i_{1}}-\zeta\right)\left(x_{i_{2}}-\zeta\right)\ldots\left(x_{i_{n-m}}-\zeta\right)(1+\theta_{5n-m+3})
\]
where the summation is over $i_{j}\in\left\{ 0,1,\ldots,n\right\} -\left\{ k\right\} $.\end{lem}
\begin{proof}
The finite difference weights are computed at $x=\zeta$ by shifting
the grid by $-\zeta$ and then using the algorithm for $x=0$. Thus
compared to the previous lemma, the subscript of $\theta$ is incremented
by $n-m$ to allow for $n-m$ subtractions inside the summation. There
is no need to redo the analysis of $w_{k}$ because $w_{k}$ is unchanged
by the shift and it is assumed that $w_{k}$ is computed prior to
shifting.
\end{proof}
The theorem below introduces $U_{\mathcal{R}}$ which is an upper
bound of the rounding error.
\begin{thm}
The magnitude of the roundoff error in the computation of the finite
difference approximation
\[
\sum_{k=0}^{n}w_{k,m}f(x_{k})
\]
to $f^{(m)}(\zeta)$ is upper bounded by 
\begin{equation}
U_{\mathcal{R}}=\gamma_{6n-m+4}|f|\sum_{k=0}^{n}m!\,|w_{k}|\, S_{n-m}\left(\left\{ |x_{0}-\zeta|,\ldots,|x_{n}-\zeta|\right\} -\left\{ |x_{k}-\zeta\right\} \right),\label{eq:bnds-UR}
\end{equation}
where $|f|$ is equal to $\max_{j}|f(x_{j})|$.\label{thm:bnds-UR}\end{thm}
\begin{proof}
For the computed value of $w_{k,m}$, we may use the previous lemma.
In forming the sum $\sum_{k=0}^{n}w_{k,m}f(x_{k})$, a total of $n+1$
terms are added and each term is formed through a single multiplication.
Therefore the computed value of $\sum_{k}w_{k,m}f(x_{k})$ is 
\[
\sum_{k=0}^{m}f(z_{k})(-1)^{n-m}m!w_{k}\,\sum_{i_{1}<\cdots<i_{n-m}}\left(x_{i_{1}}-\zeta\right)\left(x_{i_{2}}-\zeta\right)\ldots\left(x_{i_{n-m}}-\zeta\right)(1+\theta_{6n-m+4}).
\]
Here $(1+\theta_{6n-m+4})$ is obtained from $(1+\theta_{5n-m+3})(1+\theta_{n+1})$.
The upper bound is obtained by subtracting the true value of $\sum_{k}w_{k}f(x_{k})$,
taking absolute values, and using $|\theta_{6n-m+4}|\leq\gamma_{6n-m+4}$.
\end{proof}
If the weights $w_{k,m}$ are exact, except for the inevitable roundoff
in floating point representation, the computed value of $\sum_{k=0}^{n}w_{k,m}f(x_{k})$
is 
\[
\sum_{k=0}^{n}w_{k,m}f(x_{k})(1+\theta_{k+2})
\]
assuming right to left summation. Thus the magnitude of the rounding
error is bounded by 
\begin{equation}
U'_{\mathcal{R}}=|f|\sum_{k=0}^{n}|w_{k,m}|\,\gamma_{k+2}.\label{eq:bnds-URprime}
\end{equation}
For other orders of summation the $\gamma_{k+2}$ may be reordered.

\subsection{Asymptotics}

To obtain asymptotics for $U_{\mathcal{R}}$ and $U'_{\mathcal{R}}$
in the limit of increasing $n$, we introduce three quantities $\mathcal{W}_{\ell}$,
$\mathcal{E}_{m}^{\ell}$, and $\mathcal{E}_{m}^{\ell,k}$. The first
of these $\mathcal{W}_{\ell}$ is defined as $\prod_{j=0,j\neq\ell}^{n}(x_{\ell}-x_{j})$.
It is the inverse of the Lagrange weight. If $x_{0},x_{1},\ldots,x_{n}$
are the Chebyshev points, it is well-known (see \cite{DonSolomonoff1997}
for example) that 
\begin{equation}
\mathcal{W}_{\ell}=\begin{cases}
(-1)^{\ell}\frac{2n}{2^{n-1}} & \quad\text{for \ensuremath{\ell=0,n}}\\
(-1)^{\ell}\frac{n}{2^{n-1}} & \quad\text{otherwise.}
\end{cases}\label{eq:bnds-Wl}
\end{equation}
Define 
\begin{equation}
\mathcal{E}_{m}^{\ell}=\sum_{i_{1}<\cdots<i_{m}}\frac{1}{\left(x_{\ell}-x_{i_{1}}\right)\left(x_{\ell}-x_{i_{2}}\right)\cdots\left(x_{\ell}-x_{i_{m}}\right)}\label{eq:bnds-El}
\end{equation}
where $i_{j}\in\left\{ 0,1,\ldots,n\right\} -\left\{ \ell\right\} $.
If $\ell\neq k$, define 
\begin{equation}
\mathcal{E}_{m}^{\ell,k}=\sum_{i_{1}<\cdots<i_{m}}\frac{1}{\left(x_{\ell}-x_{i_{1}}\right)\left(x_{\ell}-x_{i_{2}}\right)\cdots\left(x_{\ell}-x_{i_{m}}\right)}\label{eq:bnds-Elk}
\end{equation}
where $i_{j}\in\left\{ 0,1,\ldots,n\right\} -\left\{ k,\ell\right\} $.
If $m=0$, both $\mathcal{E}_{m}^{\ell}$ and $\mathcal{E}_{m}^{\ell,k}$
are defined to be $1$.

Define 
\[
\mathcal{P}_{r}^{\ell}=\sum_{j=0,j\neq\ell}^{n}\frac{1}{(x_{\ell}-x_{j})^{r}}\quad\text{and}\quad\mathcal{P}_{r}^{\ell,k}=\mathcal{P}_{r}^{\ell}-\frac{1}{(x_{\ell}-x_{k})^{r}}.
\]
For the Chebyshev points, or indeed for any set of points, the rounding
errors are maximized at the edges as evident from inspection of (\ref{eq:bnds-UR}).
Later we will see that discretization errors too tend to be the greatest
at the edges. Therefore we set $\ell=0$, and find that 
\[
\frac{1}{x_{\ell}-x_{j}}=\frac{1}{1-x_{j}}\sim\frac{2n^{2}}{j^{2}\pi^{2}}.
\]
It follows that 
\begin{equation}
\mathcal{P}_{r}^{0}\sim\frac{2^{r}\zeta(2r)}{\pi^{2r}}n^{2r},\label{eq:bnds-Pr0-asym}
\end{equation}
and 
\begin{equation}
\mathcal{P}_{r}^{0,k}\sim\frac{2^{r}}{\pi^{2r}}n^{2r}\left(\zeta(2r)-\frac{1}{k^{2r}}\right),\label{eq:bnds-Prk-asym}
\end{equation}
where $\zeta(\cdot)$ is the zeta function. 

The Newton identities relating symmetric functions give 
\begin{align}
\mathcal{E}_{1}^{0} & =\mathcal{P}_{1}^{0}\nonumber \\
\mathcal{E}_{2}^{0} & =\mathcal{E}_{1}^{0}\mathcal{P}_{1}^{0}-\mathcal{P}_{2}^{0}\nonumber \\
\mathcal{E}_{3}^{0} & =\mathcal{E}_{2}^{0}\mathcal{P}_{1}^{0}-\mathcal{E}_{1}^{0}\mathcal{P}_{2}^{0}+\mathcal{P}_{3}^{0}.\label{eq:bnds-newton-ids}
\end{align}
Similar identities related $\mathcal{E}_{r}^{0,k}$ and $P_{r}^{0,k}$.
\begin{thm}
If $x_{0},x_{1},\ldots,x_{n}$ are the Chebyshev points, the upper
bound $U_{\mathcal{R}}$ for the rounding error with $\zeta=x_{0}=1$
has the following asymptotics in the limit of increasing $n$:
\begin{align*}
U_{\mathcal{R}} & \sim\gamma_{6n+3}|f|\left(\frac{n^{2}}{3}+\sum_{k=1}^{n}\frac{4n^{2}}{\pi^{2}}\right)=\gamma_{6n+3}|f|0.9995\ldots n^{2}\\
 & \sim\gamma_{6n+2}|f|\left(\frac{n^{4}}{30}+\sum_{k=1}^{n}\frac{4n^{4}}{\pi^{2}k^{2}}\left(\frac{1}{3}-\frac{2}{\pi^{2}k^{2}}\right)\right)=\gamma_{6n+2}|f|0.1665\ldots n^{4}\\
 & \sim\gamma_{6n+1}|f|\left(\frac{n^{6}}{630}+\sum_{k=1}^{n}\frac{2n^{6}}{15\pi^{6}k^{6}}\bigl|\pi^{4}k^{4}-20\pi^{2}k^{2}+120\bigr|\right)=\gamma_{6n+1}|f|0.01109\ldots n^{6}\\
 & \sim\gamma_{6n}|f|\left(\frac{n^{8}}{22680}+\sum_{k=1}^{n}\frac{2n^{8}}{315\pi^{8}k^{8}}\bigl|\pi^{6}k^{6}-42\pi^{4}k^{4}+840\pi^{2}k^{2}-5040\bigr|\right)=\gamma_{6n}|f|0.00039\ldots n^{8},
\end{align*}
for order of differentiation $m=1,2,3,4$, respectively. As before,
$|f|=\max_{j}f(x_{j})$.\label{thm:bnds-UR-asym}\end{thm}
\begin{proof}
If we go back to (\ref{eq:bnds-UR}), which defines $U_{\mathcal{R}}$,
and look at the $k=0$ term with $\zeta=x_{0}=1$, it can be written
as 
\begin{align*}
w_{0}S_{n-m}(1-x_{1},1-x_{2},\ldots,1-x_{n}) & =\frac{S_{n-m}(1-x_{1},1-x_{2},\ldots,1-x_{n})}{\prod_{j=1}^{n}(1-x_{j})}\\
 & =\mathcal{E}_{m}^{0}.
\end{align*}
A term with $k>0$ may be written as 
\begin{align*}
|w_{k}|S_{n-m}\left(\left\{ 0,1-x_{1},\ldots,1-x_{n}\right\} -\left\{ 1-x_{k}\right\} \right) & =\frac{S_{n-m}\left(\left\{ 1-x_{1},\ldots,1-x_{n}\right\} -\left\{ 1-x_{k}\right\} \right)}{\prod_{j=0,j\neq k}^{j=n}|x_{k}-x_{j}|}\\
 & =\frac{|\mathcal{W}_{0}|}{|\mathcal{W}_{k}|}\frac{S_{n-m}\left(\left\{ 1-x_{1},\ldots,1-x_{n}\right\} -\left\{ 1-x_{k}\right\} \right)}{\prod_{j=1}^{n}(1-x_{j})}\\
 & =\frac{|\mathcal{W}_{0}|}{|\mathcal{W}_{k}|}\frac{\mathcal{E}_{m-1}^{0,k}}{1-z_{k}}.
\end{align*}
So the summation in (\ref{eq:bnds-UR}) becomes 
\begin{equation}
\mathcal{E}_{m}^{0}+\sum_{k=1}^{n}\frac{|\mathcal{W}_{0}|}{|\mathcal{W}_{k}|}\frac{\mathcal{E}_{m-1}^{0,k}}{1-z_{k}}.\label{eq:bnds-tmp1}
\end{equation}
The proof is completed using the asymptotics for $\mathcal{P}_{r}^{0}$
and $\mathcal{P}_{r}^{0,k}$ in (\ref{eq:bnds-Pr0-asym}) and (\ref{eq:bnds-Prk-asym}),
along with Newton identities (\ref{eq:bnds-newton-ids}), to obtain
the asymptotics of $\mathcal{E}_{m}^{0}$ and $\mathcal{E}_{m-1}^{0,k}$.
These along with the formula (\ref{eq:bnds-Wl}) for $\mathcal{W}_{\ell}$
are substituted into (\ref{eq:bnds-tmp1})to obtain the asymptotics
of that quantity.
\end{proof}
The methods that utilize accurate versions of the spectral differentiation
matrix \cite{BaltenspergerTrummer2003,DonSolomonoff1997} are often
employed with $m=1$. Therefore we limit the next theorem to $m=1$.
\begin{thm}
If $m=1$ and $\zeta=x_{0}=1$, the upper bound $U_{\mathcal{R}}'$
defined by (\ref{eq:bnds-URprime}) satisfies 
\[
U'_{\mathcal{R}}\precsim2u|f|\left(\frac{n^{2}}{3}+\sum_{k=1}^{n}\frac{4n^{2}(k+2)}{\pi^{2}k^{2}}\right)\sim\frac{8}{\pi^{2}}u|f|n^{2}\log n,
\]
where $u$ is the unit-roundoff, and with the assumption $nu<1/2$.\label{thm:bnds-URp-asym}\end{thm}
\begin{proof}
If $\zeta=x_{0}=1$, the formula for $w_{k,1}$ (\ref{eq:bnds-wkm-atzero})with
$\zeta$ shifted to $0$ becomes 
\[
w_{k}S_{n-1}\left(\left\{ 0,1-x_{1},\ldots,1-x_{n}\right\} -\left\{ 1-x_{k}\right\} \right).
\]
If $k=0$, we have 
\[
w_{k,1}=\mathcal{E}_{1}^{0},
\]
and if $k>1$, we have
\[
w_{k,1}=\frac{\mathcal{W}_{0}}{\mathcal{W}_{1}}\frac{1}{1-z_{k}}\mathcal{E}_{m-1}^{\ell,k}.
\]
To complete the proof, we may obtain asymptotics for $w_{k,1}$ as
in the previous proof and use $\gamma_{k+1}\leq2(k+1)u$ which holds
under the assumption $nu<1/2$.
\end{proof}
Comparison of Theorems \ref{thm:bnds-UR-asym} and \ref{thm:bnds-URp-asym}
gives an indication of the advantage obtained by computing the weights
$w_{k,m}$ accurately followed by careful summation. Since $\gamma_{n}\approx nu$,
the bound for the first derivative in Theorem \ref{thm:bnds-UR-asym}
increases at the rate $n^{3}$. In Theorem \ref{thm:bnds-URp-asym},
the rate is $n^{2}\log n$. Thus an $n$ is replaced by $\log n$.
This is very similar to the advantage obtained using compensated summation
and other methods of precise summation \cite{Higham2002}. The comparison
also shows that a sound method for calculating the weights $w_{k,m}$
introduces only a modest amount of error.

\section{Discretization error }

In this section, we will give a discussion of the discretization error.
We show that the discretization error goes up a factor of $n^{2}$
with every additional derivative just like the rounding error. This
implies that the value of $n$ where the total error transitions from
mostly due to discretization to mostly due to rounding is independent
of the order of the derivative. This implication is illustrated computationally.

The Lagrange interpolant may be augmented with the remainder term
as follows \cite{ContedeBoor1980,Davis1975}:
\[
f(x)=\sum_{k=0}^{n}f(x_{k})\ell_{k}(x)+f[x_{0},x_{1},\ldots,x_{n},x]\prod_{k=0}^{n}(x-x_{j}).
\]
Here the $f[]$ notation is for divided differences. The finite difference
approximation to $f^{(m)}(x)$ is obtained by differentiating the
Langrange interpolant $m$ times. Therefore the discretization error
for the $m$-th derivative at $x=\zeta$ is equal to 
\[
\frac{d^{m}}{d\zeta^{m}}f[x_{0},\ldots,x_{n},\zeta](\zeta-x_{0})(\zeta-x_{1})\ldots(\zeta-x_{n}).
\]
The product rule for differentiation gives
\begin{equation}
\sum_{j=0}^{m}\binom{m}{j}\frac{d^{j}}{d\zeta^{j}}f[z_{0},\ldots,z_{n},\zeta]\:(m-j)!\, S_{n+1-m+j}\left(\zeta-x_{0},\ldots,\zeta-x_{n}\right).\label{eq:bnds-discerr-divdiff'}
\end{equation}
Using standard properties of divided differences, and assuming $f$
to be differentiable as many times as necessary, we may write the
discretization error as 
\[
\sum_{j=0}^{m}m!f[x_{0},\ldots,x_{n},\zeta^{(j+1)}]\, S_{n+1-m+j}\left(\zeta-x_{0},\ldots,\zeta-x_{n}\right),
\]
 where $\zeta^{(j+1)}$ stands for $\zeta$ repeated $j+1$ times
in the divided difference. Here we have used an identity for differentiating
a divide difference \cite{ContedeBoor1980}. If $x_{0},\ldots,x_{n}$
are the Chebyshev points and the divided differences are assumed to
be relatively uniform throughout the domain, this expression above
shows that the discretization error too is likely to be maximum at
the edges. Therefore we take $\zeta=x_{0}=1$ to get 
\[
U_{\mathcal{D}}=\sum_{j=0}^{m}m!f[1^{(j+2)},x_{1},\ldots x_{n}]\, S_{n+1-m+j}(0,1-x_{1},\ldots,1-x_{n}).
\]
We will denote the divided difference $f[1^{(j+2)},x_{1},\ldots x_{n}]$
by $D_{j+2}$. The expression for the discretization error becomes
\begin{equation}
U_{\mathcal{D}}=\sum_{j=0}^{m}m!D_{j+2}S_{n+1-m+j}(1-x_{1},\ldots,1-x_{n}).\label{eq:disc-UD}
\end{equation}

As $n$ increases, the asymptotics of $U_{\mathcal{D}}$ are given
by 
\begin{align}
U_{\mathcal{D}} & \sim4\left(\frac{D_{2}}{n2^{n}}\right)n^{2}\nonumber \\
 & \sim\frac{8}{3}\left(\frac{D_{2}}{n2^{n}}\right)n^{4}+\frac{8nD_{3}}{2^{n}}\nonumber \\
 & \sim\frac{4}{5}\left(\frac{D_{2}}{n2^{n}}\right)n^{6}+\frac{8n^{3}D_{3}}{2^{n}}+\frac{24nD_{4}}{2^{n}}\nonumber \\
 & \sim\frac{16}{105}\left(\frac{D_{2}}{n2^{n}}\right)n^{8}+\frac{16n^{5}D_{3}}{5.2^{n}}+\frac{32n^{3}D_{4}}{2^{n}}+\frac{96nD_{5}}{2^{n}}\label{eq:disc-UD-asym}
\end{align}
for orders of differentiation $m=1,2,3,4$, respectively. The symmetric
function $S_{n+1-m+j}$ in (\ref{eq:disc-UD}) is equal to $\mathcal{W}_{0}\mathcal{E}_{m-j}^{0}$.
From this point the symmetric functions may be estimated as in the
previous section to derive (\ref{eq:disc-UD-asym}). 

We do not attempt to estimate the divided differences $D_{j+2}$.
However they may be estimated using contour integration as shown in
\cite{WeidemanTrefethen1988}. If $f(x)=\sin Kx$, and $K=n\pi/\eta$
with $\eta>\pi$, implying more than $\pi$ points per wavelength,
the divided difference $D_{2}$ decreases exponentially with $n$.
On the other hand, if $K$ has a fixed value such as $K=2\pi$ the
divided difference decreases super-exponentially with $n$. For functions
such as $f(x)=\sin\pi x$, the divided differences $D_{1},$$D_{2}$,
and so on typically vary only by constant factors and the asymptotics
in (\ref{eq:disc-UD-asym}) may be expected to be dominated by the
$D_{2}$ term.

The interpolation error at $x=\zeta$ is given by 
\[
f[x_{1},x_{2},x_{3},\ldots,x_{n},\zeta](\zeta-x_{1})\ldots(\zeta-x_{n}).
\]
Comparison with (\ref{eq:bnds-Wl}) shows that the interpolation error
is approximately
\begin{equation}
f[x_{1},x_{2},x_{3},\ldots,x_{n},\zeta]\frac{C}{n2^{n}}\label{eq:disc-interp-error}
\end{equation}
for $\zeta\in(x_{1},x_{0})$ and where $C$ is $\mathcal{O}(1)$. 

Comparison of Theorem \ref{thm:bnds-UR-asym} and (\ref{eq:disc-UD-asym})
suggests that the transition from discretization error to rounding
error should be relatively independent of the order of the derivative.
Every time the order goes up by $1$, both estimates increases by
a factor $n^{2}$, ignoring constants. Therefore the transition should
be at about the same value of $n$ independently of the order of differentiation.
This phenomenon is illustrated in Figure \ref{fig:disc-rerr-transition}.

\begin{figure}
\begin{centering}
\includegraphics[scale=0.35]{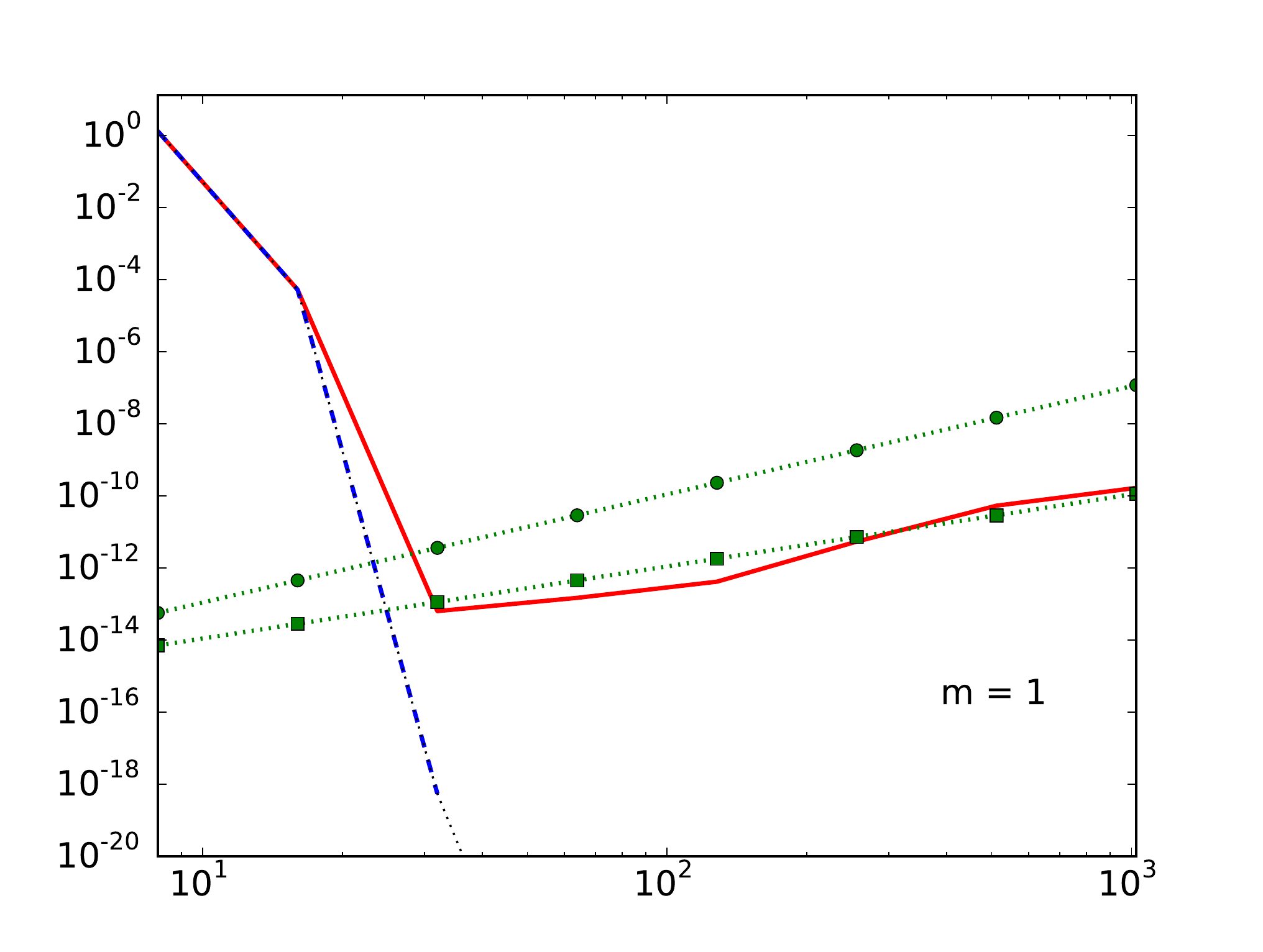}\includegraphics[scale=0.35]{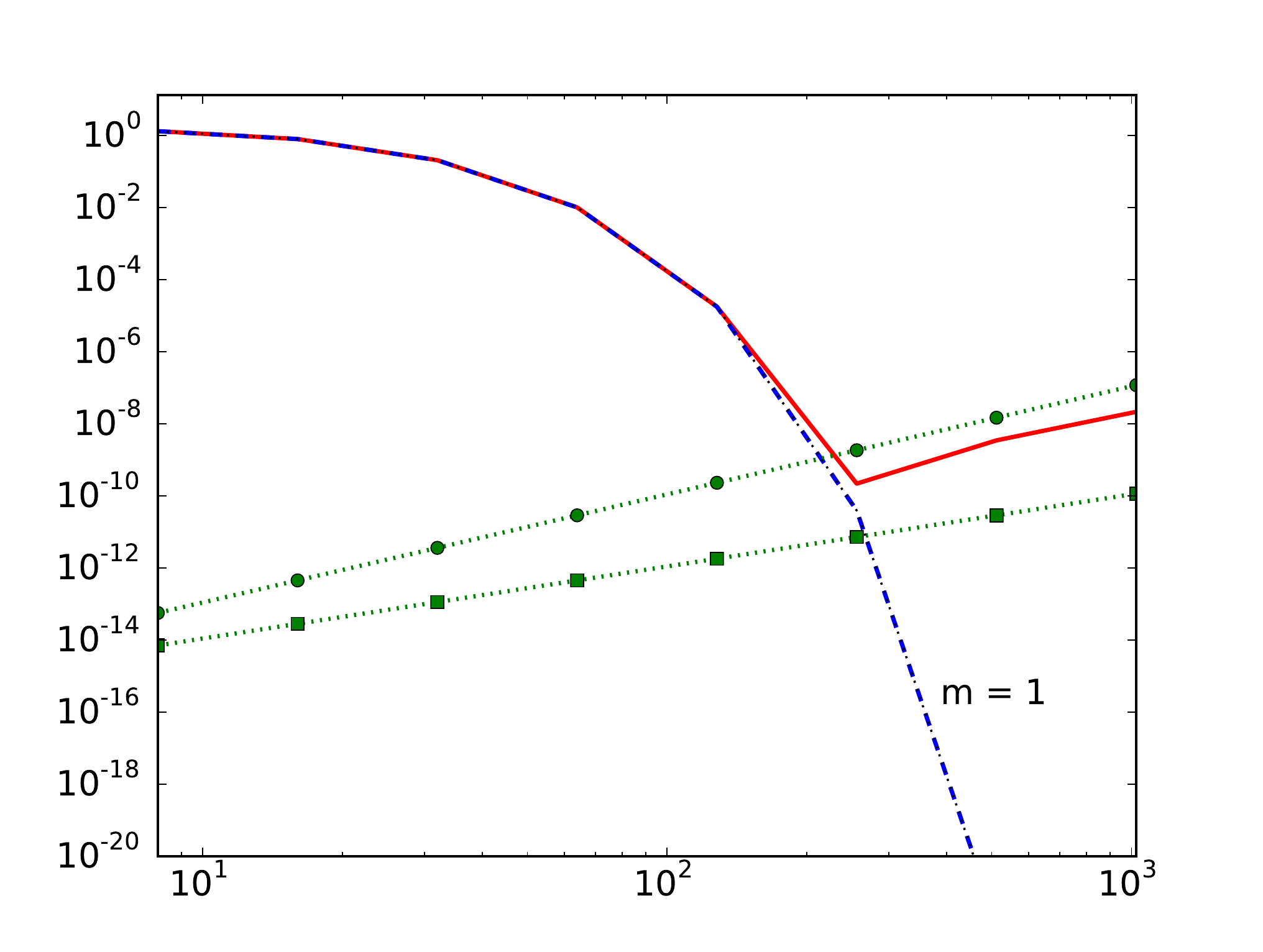}
\par\end{centering}

\begin{centering}
\includegraphics[scale=0.35]{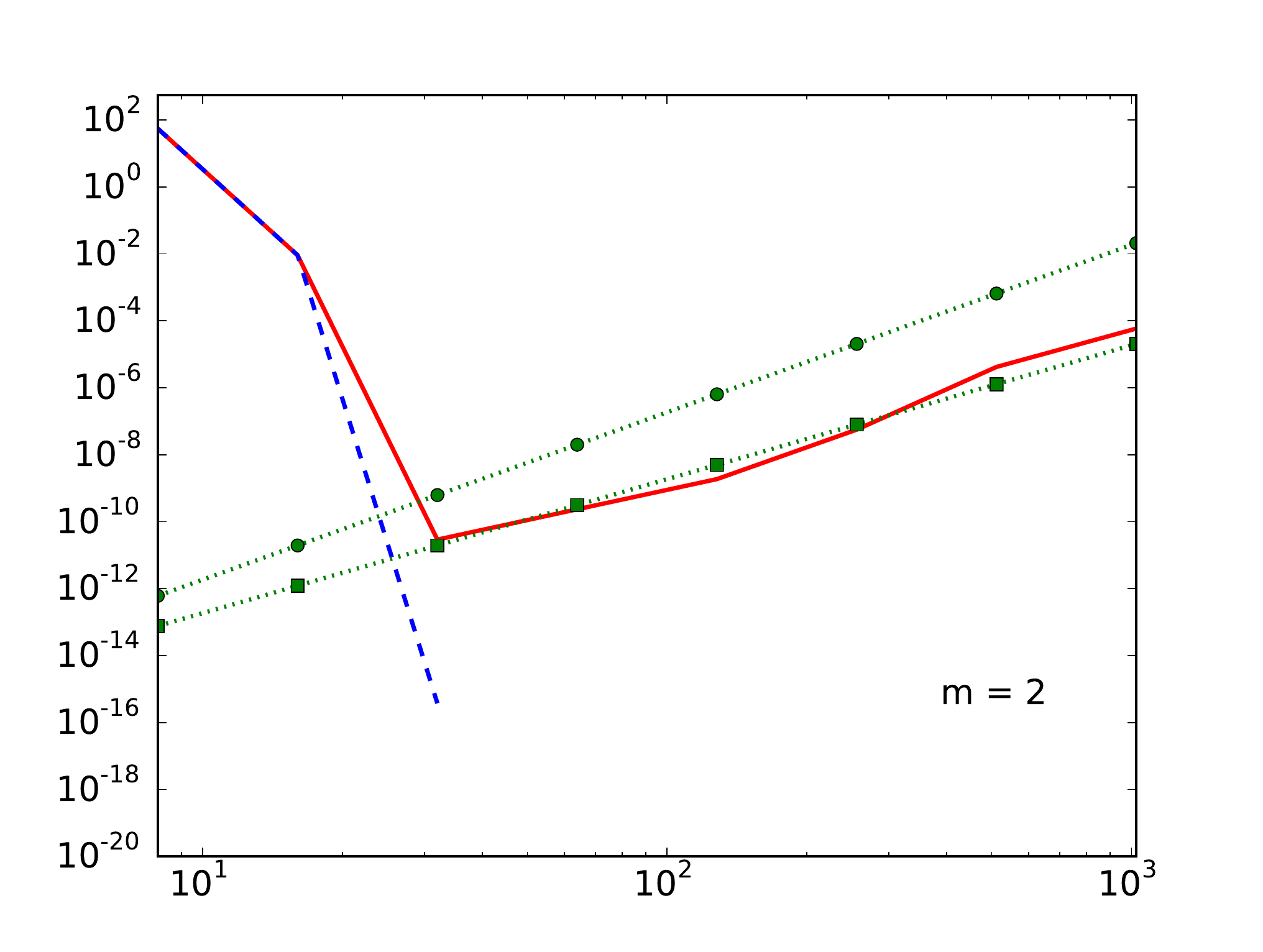}\includegraphics[scale=0.35]{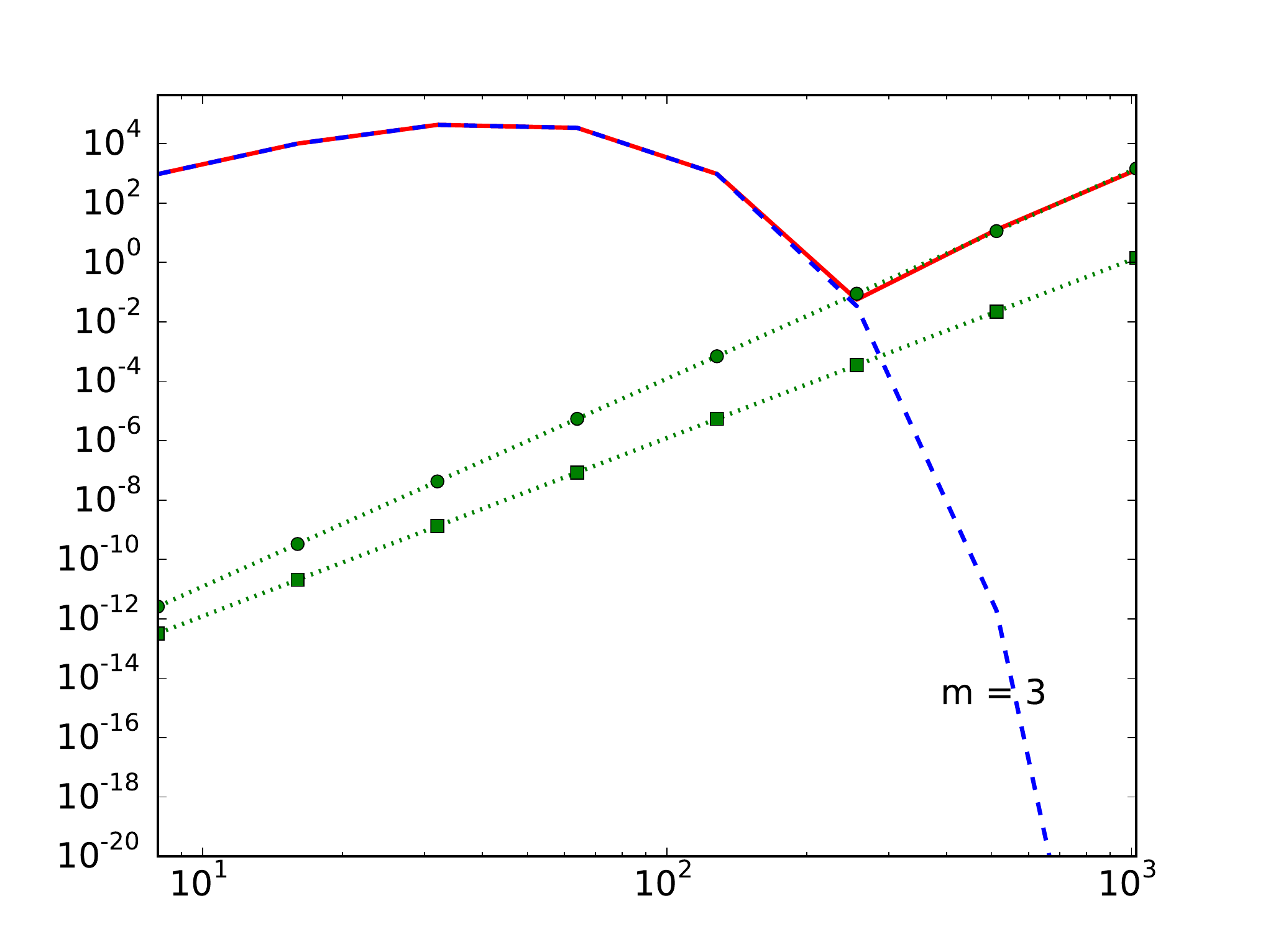}
\par\end{centering}

\begin{centering}
\includegraphics[scale=0.35]{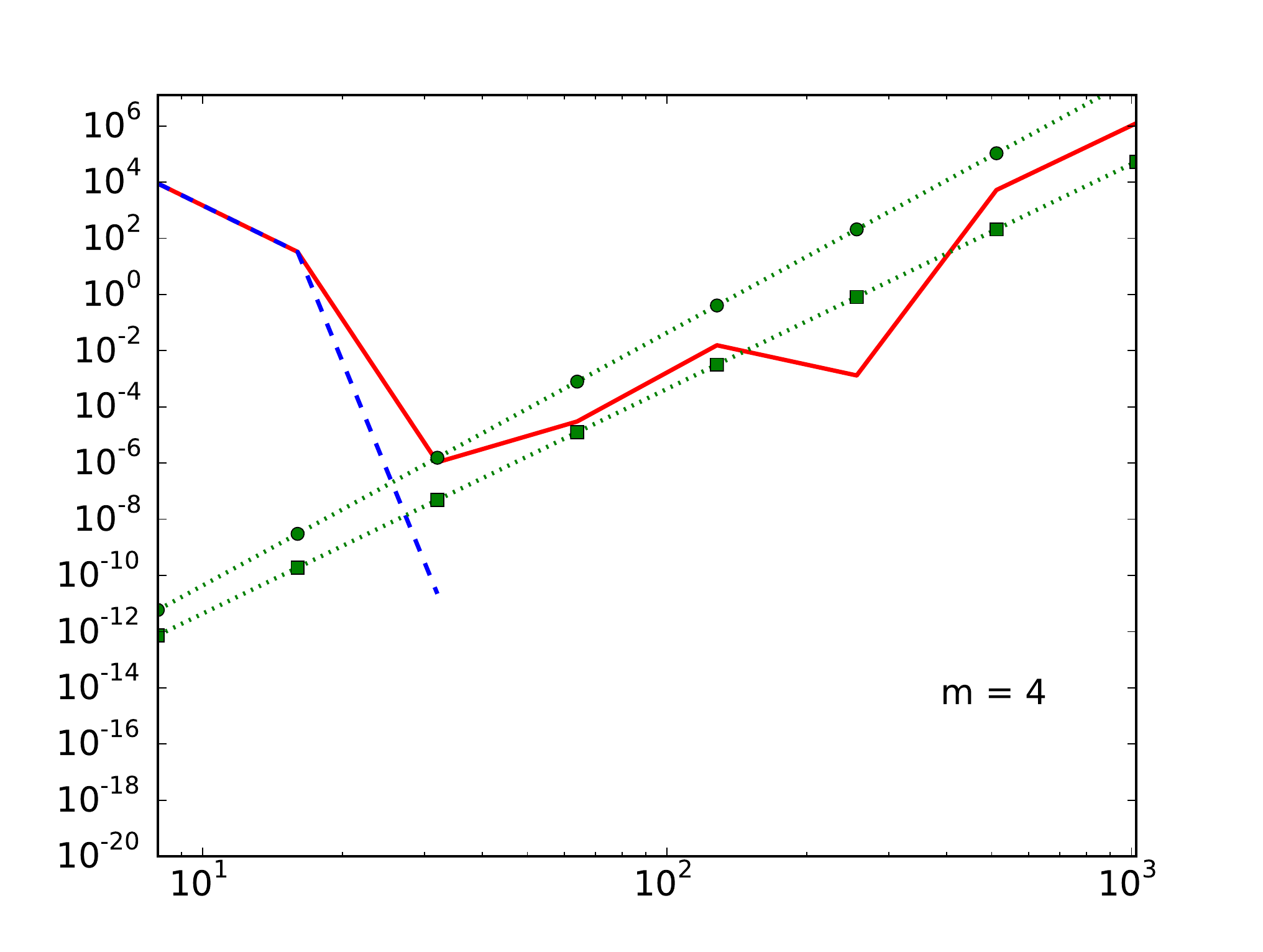}\includegraphics[scale=0.35]{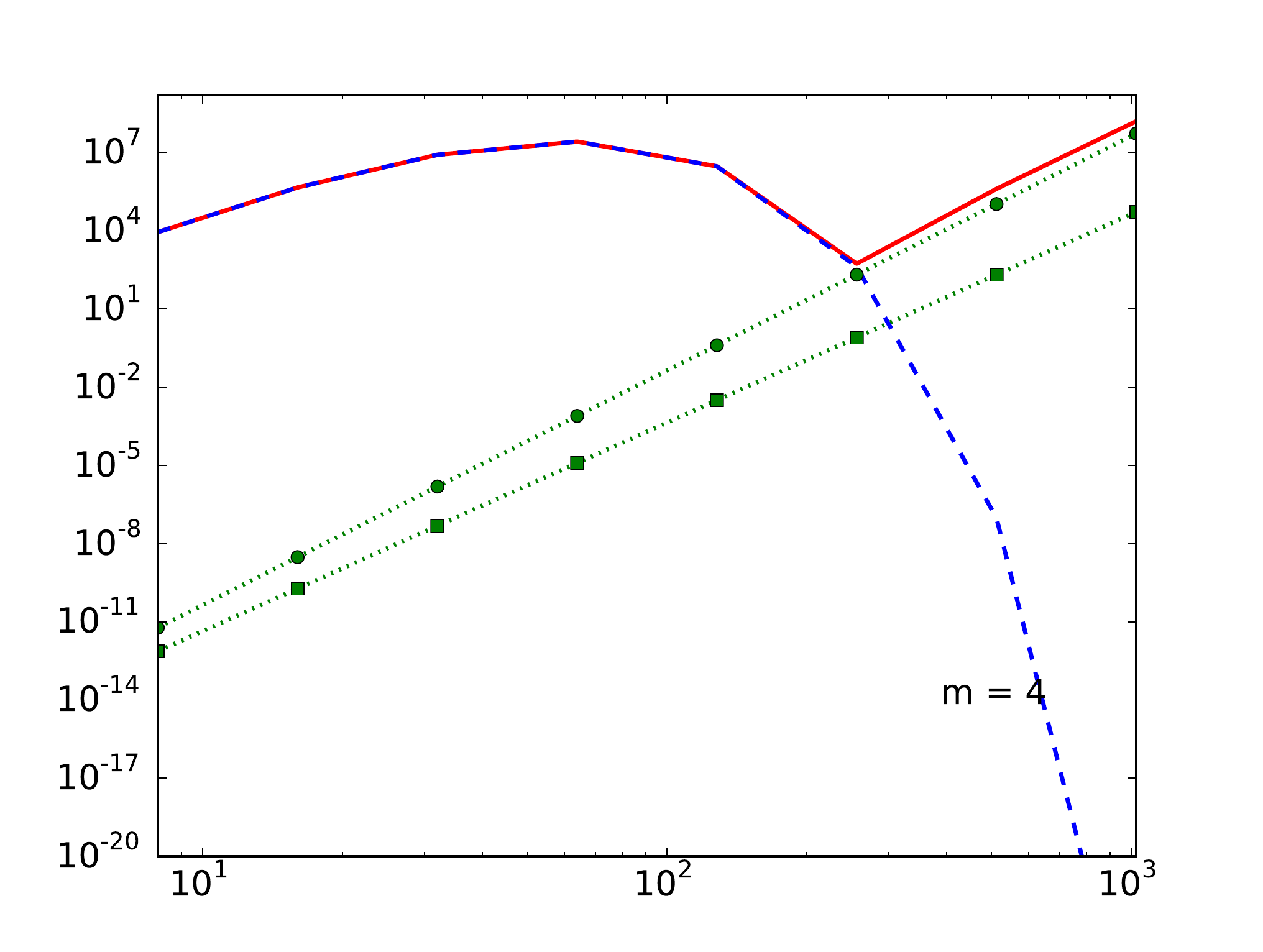}\caption{Plots of error vs $n$. The plots on the left are for $f(x)=\sin2\pi x$.
The plots on the right are for $f(x)=\sin Kx$ with $K=n\pi/4$ implying
$4$ points per wavelength. The solid line is the actual error. The
dashed line is the discretization error computed using (\ref{eq:bnds-discerr-divdiff'}),
with divided differences computed in extended precision. The dotted
lines are the asymptotic rounding error bounds of Theorem \ref{thm:bnds-UR-asym}.
The dotted lines with circles replace $\gamma_{6n+4-m}$ by $nu$
and the dotted lines with squares replace that quantity by $u$.\label{fig:disc-rerr-transition}}

\par\end{centering}

\end{figure}

\section{Choice of the mapping parameter}

The choice of the parameter $\alpha$ is taken to be given by 
\begin{equation}
\left(\frac{1-\sqrt{1-\alpha^{2}}}{\alpha}\right)^{n}=n^{\beta}u\label{eq:choice-balance-eqn}
\end{equation}
with $\beta=0$ \cite{DonSolomonoff1997,KosloffTalEzer1993} and with
$u$ being the unit roundoff. We will attempt to justify this choice
for all orders of derivative $m$.

Given a function $f(x)$, such as $f(x)=\sin Kx$, the mapped function
is $F(\xi)=f(g(\xi))$ where $g(\cdot)$ is the mapping (\ref{eq:intro-map-fn-g}).
We will first argue for (\ref{eq:choice-balance-eqn}) as a balance
between the discretization error and the rounding error in interpolation.
The analysis of rounding error that arises in spectral differentiation
is precise. The indeterminacy in the rounding error is limited to
a factor of $n$, as may be seen from Figure \ref{fig:disc-rerr-transition}.
However, the discretization errors cannot be estimated as precisely
because the divided differences that arise in (\ref{eq:disc-UD-asym})
and (\ref{eq:disc-interp-error}) are not known within factors of
$n$.

If $f(x)=1$ then $F(\xi)=g(\xi)=\arcsin\alpha\xi/\arcsin\alpha$.
The Chebyshev series of $F(\xi$) may be computed from the Laurent
series of $F((z+1/z)/2)$ centered at $z=0$ (if $z=e^{i\theta}$
then $\xi=\cos\theta$, and $(z^{n}+1/z^{n})/2=\cos n\theta$ is the
Chebyshev polynomial $T_{n}(\xi)$. If $F(\xi)=g(\xi)$, the singularities
are at 
\[
z=\frac{\pm1\pm\sqrt{1-\alpha^{2}}}{\alpha}.
\]
Therefore the coefficients of $z^{\pm n}$ in the Laurent series fall
off in magnitude at the rate 
\[
\left(\frac{1-\sqrt{1-\alpha^{2}}}{\alpha}\right)^{n}
\]
and so does the coefficient of $T_{n}(\xi)$ in the Chebyshev series
of $g(\xi)$. In fact, one can be more precise. Because the singularities
of $g(\xi)$ at $\xi=\pm1/\alpha$ are of the type $(\xi\pm1/\alpha)^{1/2},$
the coefficients will fall off at the rate 
\begin{equation}
n^{-3/2}\left(\frac{1-\sqrt{1-\alpha^{2}}}{\alpha}\right)^{n}.\label{eq:choice-interp-err}
\end{equation}
This may be taken as an estimate of the discretization error in $g(\xi)$. 

When $f(x)=\sin Kx$, the estimate (\ref{eq:choice-interp-err}) for
interpolation error will still hold but with additional modulation
factors of the type $n^{\beta}$ with $\beta>0$. These modulating
factors are not precisely known but they certainly exist. For example,
if $K=n\pi/4$, implying $4$ points per wavelength, the number of
terms in the expansion of $\sin Kg(\xi)$ (in powers of $g(\xi)$)
before the exponentially decay of coefficients kicks in, is greater
than $\mathcal{O}(n)$.

As far as the rounding error in interpolation is concerned, this quantity
is bounded by $Cn\log nu$, with $C$ being a small constant \cite{Higham2004}.
Thus balancing of discretization and rounding errors leads to (\ref{eq:choice-interp-err})
but with an indeterminacy in the exact value of $\beta$ and in constants.
The appropriate balance, ignoring constants, is given by (\ref{eq:choice-balance-eqn}).
Empirically, $\beta=0$ is found to be a good choice although other
$\beta$ such as $\beta=-1.5$ seem to do just as well. See Figure
\ref{fig:choice-varybeta}.

\begin{figure}

\centering{}\includegraphics[scale=0.35]{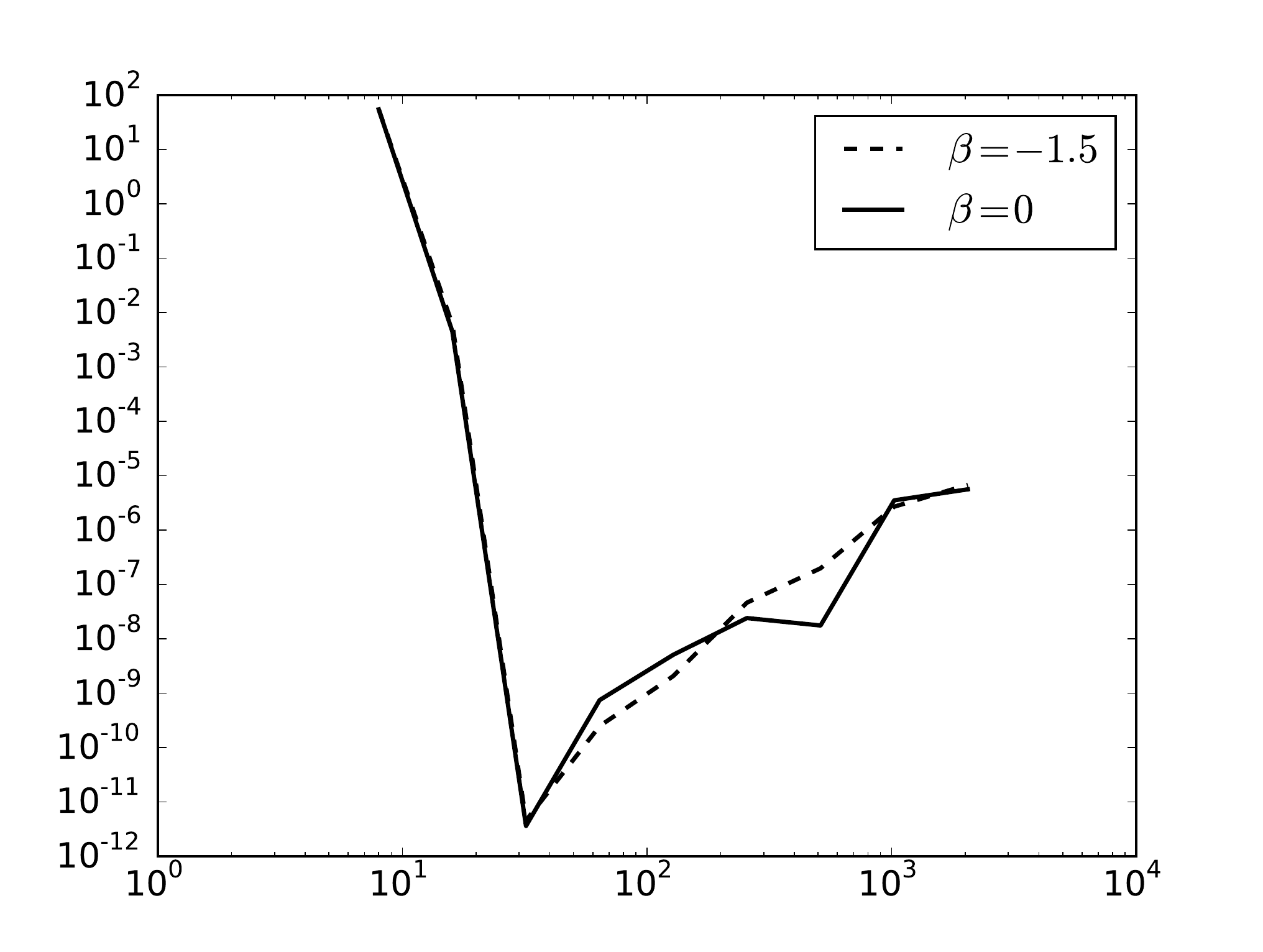}\includegraphics[scale=0.35]{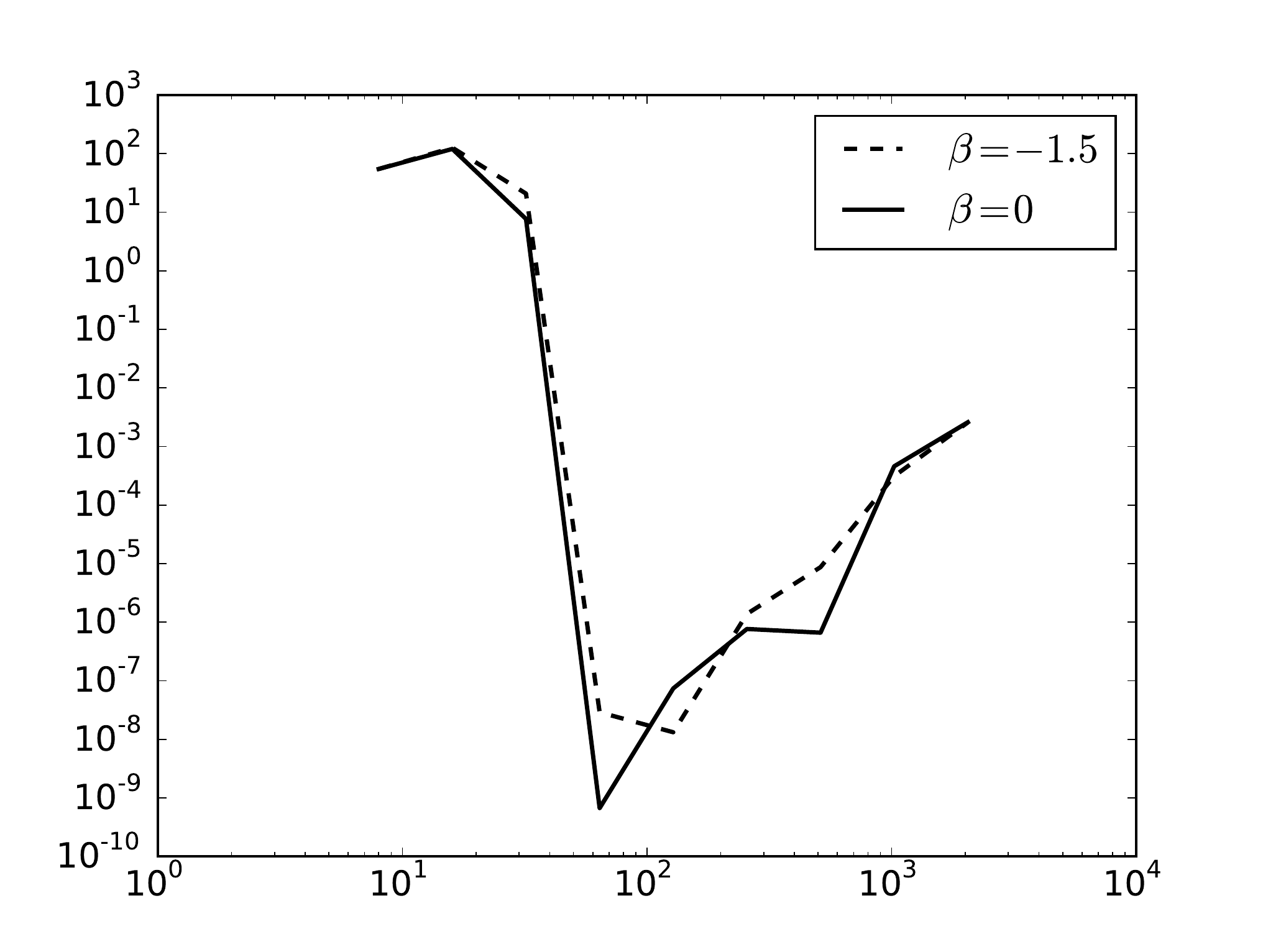}\caption{Graphs of error vs $n$ for $\sin2\pi x$ and $\sin n\pi x/4$. The
mapping parameter $\alpha$ is determined using (\ref{eq:choice-balance-eqn}).
The errors are for the $2$nd derivative.\label{fig:choice-varybeta}}
\end{figure}

As long as constants are ignored, (\ref{eq:choice-interp-err}) remains
the right equation for balancing errors for the first derivative as
well. The derivative $F'(\xi)$ is approximated by spectral differencing
at the Chebyshev points. The discretization error as well as the interpolation
error at the edge $\xi=1$ go up by a factor of $n^{2}$ from Theorem
\ref{thm:bnds-UR-asym}, (\ref{eq:disc-UD-asym}), and (\ref{eq:disc-interp-error}).
The errors are pulled back into the $x$-domain through the same $g^{-1}$
transformation, and the balancing equation remains the same.

For higher derivatives $F^{(m)}(\xi)$ the balancing equation again
remains the same, ignoring constants. With every increase in $m$
by $1$, the discretization and rounding errors both go up by a factor
of $n^{2}$. Both errors are pulled back using the same transformation
$g^{-1}$. If derivatives $f^{(m)}(x)$ are computed by successively
taking the first derivative (as in \cite{DonSolomonoff1997}), rather
than using a differencing scheme for the $m$-th derivative directly,
the argument changes only slightly. 

The discrete cosine transform is a faster method of approximating
$F'(\xi$). However, it appears to incur greater rounding error \cite{DonSolomonoff1997}.
This suggests trying to balance errors in (\ref{eq:choice-balance-eqn})
with $\beta>0$, as the greater error can only be due to rounding.
In Figure (\ref{fig:choice-dct}), $\beta=0.5$ does give smaller
errors for $f(x)=\sin2\pi x$ and the rounding errors vary more smoothly
for $f(x)=\sin n\pi x/4$.

\begin{figure}
\centering{}\includegraphics[scale=0.35]{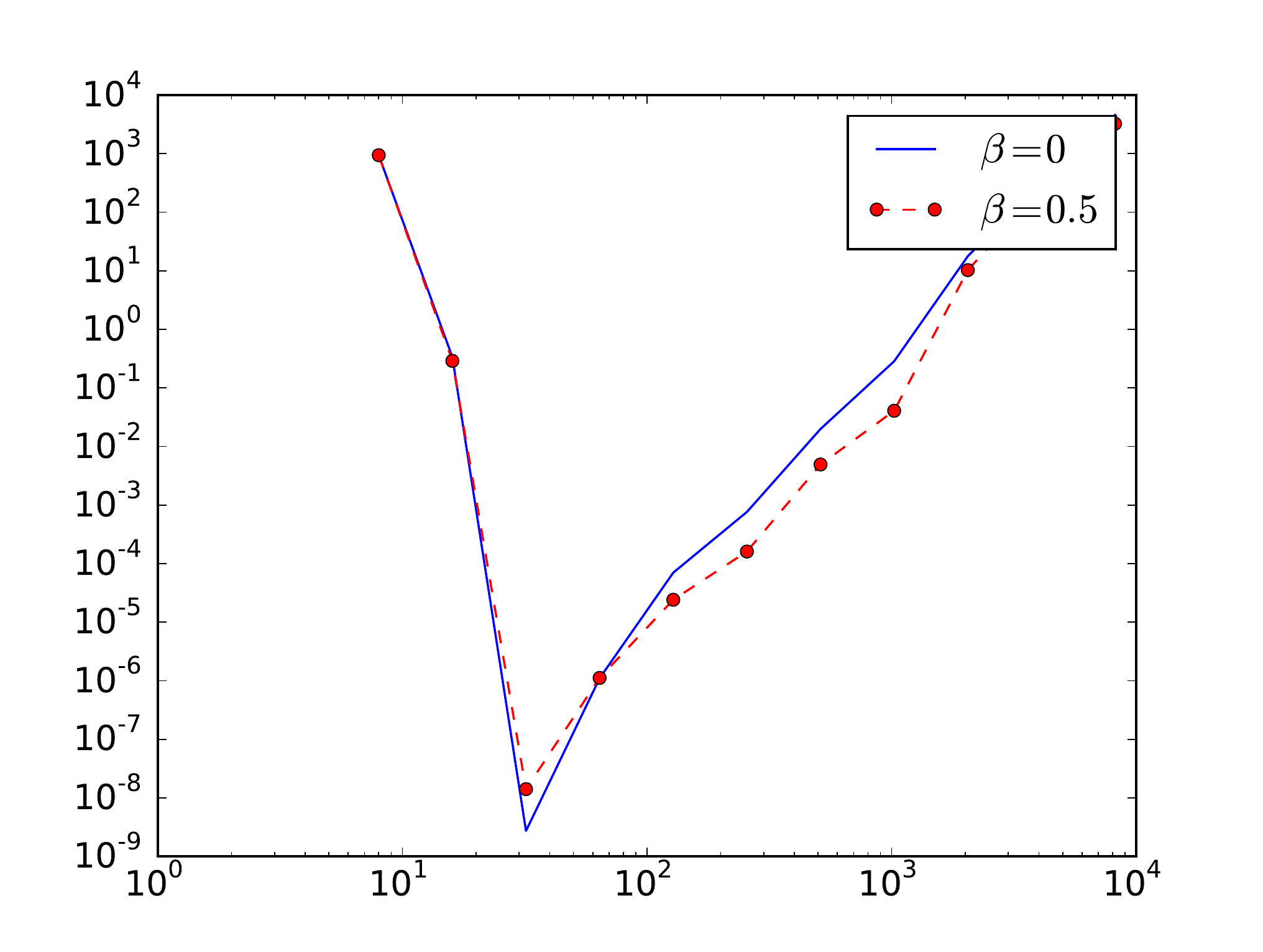}\includegraphics[scale=0.35]{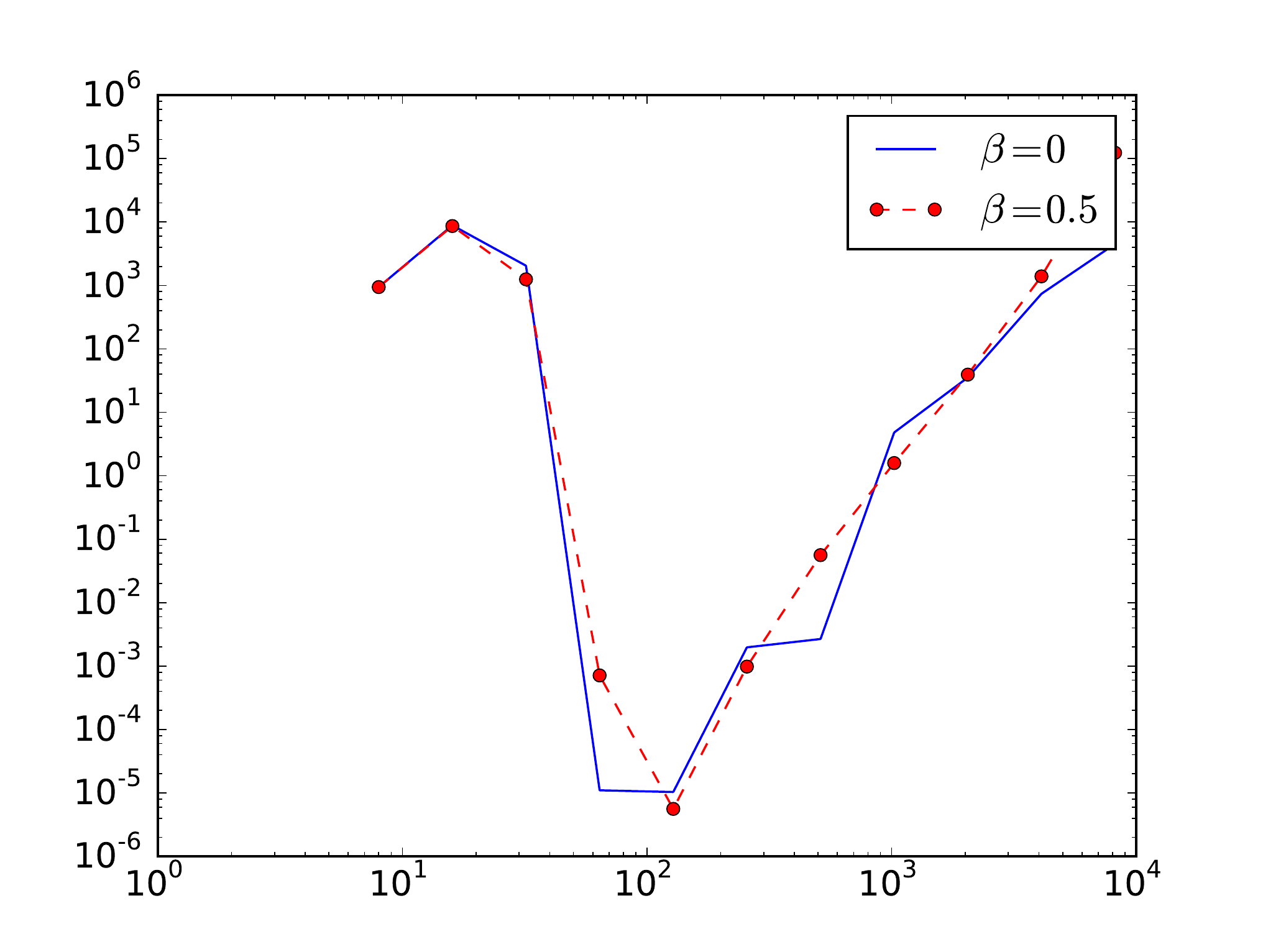}\caption{Graphs of error vs $n$ for $\sin2\pi x$ and $\sin n\pi x/4$. The
mapping parameter $\alpha$ is determined using (\ref{eq:choice-balance-eqn}).
The errors are for the $3$rd derivative.\label{fig:choice-dct}}
\end{figure}

\section{Acknowledgements}

I am very grateful to Hans Johnston for many helpful discussions.
This research was partially supported by NSF grant DMS-1115277. 

\bibliographystyle{plain}
\bibliography{references}

\end{document}